\documentclass[12pt]{amsart}
\usepackage[top=30truemm,bottom=30truemm,left=25truemm,right=25truemm]{geometry}
\usepackage{CJK}
\usepackage{mathrsfs}
\usepackage{tikz-cd}
\usepackage{color}
\usepackage{bm}
\usepackage{amsfonts,amssymb}
\usepackage{dsfont}
\usepackage{amscd}
\usepackage{extarrows}
\usepackage{amsmath}
\usepackage{mathrsfs}
\usepackage{enumerate}
\usepackage{amscd}
\usepackage[all]{xy}
\usepackage[pagebackref,colorlinks]{hyperref}
\UseRawInputEncoding
\usepackage{mathrsfs}
\usepackage{bm}
\usepackage{amsfonts,amssymb}
\usepackage{dsfont}
\usepackage{extarrows}
\usepackage{amsmath}
\usepackage{mathrsfs}
\usepackage{enumerate}
\usepackage{amscd}
\usepackage[all]{xy}
\usepackage{array}
\usepackage{arydshln}
\usepackage{hyperref}

\theoremstyle{plain} 

\theoremstyle{definition} 

\newtheorem{thm}{Theorem}[section]

\newtheorem{lem}[thm]{Lemma}

\theoremstyle{definition}

\theoremstyle{remark}

\newcommand{\be}{\begin{equation}}
	\newcommand{\ee}{\end{equation}}
\newcommand{\bea}{\begin{eqnarray}}
	\newcommand{\eea}{\end{eqnarray}}
\newcommand{\ben}{\begin{eqnarray*}}
	\newcommand{\een}{\end{eqnarray*}}
\newcommand{\bt}{\begin{split}}
	\newcommand{\et}{\end{split}}
\newcommand{\bet}{\begin{equation}}
	\newcommand{\mc}{\mathbb{C}}
	
	\newcommand{\ra}{\rightarrow}
\begin{document}
		\title[]{Characterization of negative line bundles whose Grauert blow-down are quadratic transforms}
		
		\author[F. Deng]{Fusheng Deng}
		\address{Fusheng Deng: \ School of Mathematical Sciences, University of Chinese Academy of Sciences\\ Beijing 100049, P. R. China}
		\email{fshdeng@ucas.ac.cn}
		%
	     \author[Y. Li]{Yinji Li}
	     \address{Yinji Li:  Institute of Mathematics\\Academy of Mathematics and Systems Science\\Chinese Academy of
		Sciences\\Beijing\\100190\\P. R. China}
	     \email{liyinji@amss.ac.cn}
           \author[Q. Liu]{Qunhuan Liu}
           \address{Qunhuan Liu: \ School of Mathematical Sciences, University of Chinese Academy of Sciences\\ Beijing 100049, P. R. China}
           \email{liuqunhuan23@mails.ucas.edu.cn}
		\author[X. Zhou]{Xiangyu Zhou}
		\address{Xiangyu Zhou: Institute of Mathematics\\Academy of Mathematics and Systems Science\\and Hua Loo-Keng Key
			Laboratory of Mathematics\\Chinese Academy of
			Sciences\\Beijing\\100190\\P. R. China}
		\address{School of
			Mathematical Sciences, University of Chinese Academy of Sciences,
			Beijing 100049, P. R. China}
		\email{xyzhou@math.ac.cn}
		
		\begin{abstract}
We show that the Grauert blow-down of a holomorphic negative line bundle $L$ over a compact complex space is a quadratic transform if and only if $k_0L^*$ is very ample and $(k_0+1)L^*$ is globally generated,
where $k_0$ is the initial order of $L^*$, namely, the minimal integer such that $k_0L^*$ has nontrivial holomorphic section.
		\end{abstract}
		
		\thanks{}

		\maketitle

\section{Introductions}
In the present article, we assume all complex spaces $X$ are reduced and connected.

Let $L$ be a holomorphic line bundle on $X$.
$L$ is called negative (in the sense of Grauert) if the zero section $X$ is exceptional in $L$.
Namely, there is a complex space $Z$ and a proper holomorphic map $f:L\ra Z$ such that $X$ is mapped to a single point $z_0$, 
$f:L\backslash X\ra Z\backslash\{z_0\}$ is biholomorphic, and for every open subset $U\subseteq Z$ and every holomorphic function $h$ on $V=f^{-1}(U)$, 
there is a holomorphic function $g$ on $U$ such that $h=g\circ f$. 
By a remarkable result of Grauert \cite{Gr62}, $L$ is negative if and only if $L^*$ is ample.
Such a pointed complex space $(Z,z_0)$ is unique up to biholomorphic equivalence.
We call the map $f:L\ra Z$ the \emph{Grauert blow-down} of $L$.

In this article, we study the natural question: \emph{when is $f:L\ra Z$ obtained as the quadratic transform of $Z$ with center $z_0$}?
Here we recall the notion of quadratic transform.
Let $\mathfrak{m}_{z_0}$ be the maximal ideal of the local ring $\mathcal O_{Z,z_0}$ of holomorphic function germs on $Z$ at $z_0$,
and let $f_1, \cdots, f_N$ form a set of generators of $\mathfrak{m}_{z_0}$. 
Since the construction is local in nature, we may assume $f_1,\cdots, f_N\in \mathcal O(Z)$.
Then the quadratic transform of $Z$ with center $z_0$ is defined to be the closure $Bl_{z_0}(Z)$ of the set
$$\{(z,[f_1(z):\cdots:f_N(z)])\in Z\times\mathbb P^{N-1}; z\in Z, z\neq z_0\}$$
in $Z\times\mathbb P^{N-1}$, together with the map $\pi:Bl_{z_0}(Z)\ra Z$ given by the restriction of the natural projection from $Z\times\mathbb P^{N-1}$ to $Z$.
One can show that the definition does not depend on the choice of generators, up to isomorphism in the obvious sense.
Sometimes one also calls $\pi:Bl_{z_0}(Z)\ra Z$ the blow up of $Z$ along the ideal $\mathfrak{m}_{z_0}$ (this is a special case of blow up along general coherent ideal sheaves).

We also recall the notion of the \emph{initial order} of a line bundle, which is introduced in \cite{DLLZ24}:
for holomorphic line bundle $L$ on a compact complex space $X$,
the initial order of $L$ is the minimal positive integer $k_0$ such that $k_0L$ has non-zero global section.
Our main theorem is as follows.
\begin{thm}\label{thm:main}
Let $L$ be a negative line bundle over a compact complex space and 
$f:(L,X)\ra(Z,z_0)$ be the Grauert blow-down.
Then $f:L\ra Z$ is the quadratic transform of $Z$ with center $z_0$ if and only if $k_0L^*$ is very ample and $(k_0+1)L^*$ is globally generated, where $k_0$ is the initial order of $L^*$.
\end{thm} 

Recall that the theta characteristic of a compact Riemannian surface $X$ is a line bundle $\xi$ on $X$ such that $2\xi$ is isomorphic to the canonical line bundle of $X$.
It is known that for generic compact Riemannian surfaces of genus $\geq 3$, there exists theta characteristics satisfying the condition in Theorem \ref{thm:main} but do not have nontrivial section (see \S \ref{sec:theta char} for details),
so they provide a family of counterexamples to a theorem of Morrow and Rossi \cite[Theorem 5.6]{MR79} 
which states that the Grauert blow-down of a negative line bundle over a compact complex manifold is a quadratic transform if and only if its dual bundle $L^*$ is very ample
(indeed this statement is a direct corollary of the quoted theorem).

\emph{Remark.}
We do not know if the second condition in Theorem \ref{thm:main} is redundant,
namely, if the very ampleness of $k_0L$ automatically implies that $(k_0+1)L$ is globally generated.

 \subsection*{Acknowledgements}
  This research is supported by National Key R\&D Program of
  China (No. 2021YFA1003100), NSFC grants (No. 12471079), and the Fundamental Research Funds for the Central Universities.

\section{Proof of the main theorem}
\subsection{Preliminaries}
In this subsection, we recall some well known basic facts about holomorphic line bundles that are related our discussion.

Let $p:L\ra X$ be a holomorphic line bundle over a  compact complex space $X$.
Given a holomorphic section $f$ of the dual bundle $L^*$ of $L$ and $x\in X$, $f_x\in L^*_x$ is a linear function on $L_x$.
As $x$ varies over $X$, $f$ gives a holomorphic function on the total space of $L$ which is linear along fibers.
In general, for any $k\geq 1$, a section $f\in H^0(X,kL^*)$ can be viewed as a holomorphic function on $L$,
whose restriction on each fiber is a homogenous polynomial of degree $k$.
So we can view 
$$R(X,L^*):=\bigoplus_{k=0}^{\infty}H^0(X,kL^*)$$
as a subalgebra of the $\mc$-algebra of holomorphic functions on $L$.

On the other hand, we consider the natural circular group $S^1$ action on $L$ which is given by scalar product along fibers.
Then any holomorphic function on a $S^1$-invariant neighborhood $U$ of the zero section $X$ in $L$ can be represented as a series 
whose components are given by holomorphic sections of $kL^*$, $(k\geq 0)$.
In fact, viewed as a linear subspace of $\mathcal O(U)$, $H^0(X,kL^*)$ can be characterized as
$$H^0(X,kL^*)=\{f\in\mathcal O(U); f(\alpha z)=\alpha^kf(z)\ \text{for}\ z\in U, \alpha\in S^1\}.$$
We summarize this as the following lemma.

\begin{lem}\label{lem:expansion}
Suppose $X$ is a  compact complex space, $L$ is a holomorphic line bundle on $X$, and $U$ is a connected $S^1$-invariant neighborhood of the zero section $X$.
Then every holomorphic function $h$ on $U$ can be written as a series
\begin{align*}
h=\sum_{k=0}^{\infty}h_k,\ h_k\in H^0(X,kL^*)
\end{align*}
which converges uniformly on compact subsets of $U$.
\end{lem}
\begin{proof}
This is a consequence from basic representation theory of compact Lie groups, 
but we give a direct argument for it here.
Let $\{U_i\}_{i\in I}$ be an open covering of $X$ such that $L|_{U_i}\cong U_i\times \mc$.
In each $U_i\times\mc$, 
we can develop $h$ into a locally uniformly convergence series 
\begin{align*}
h(x_i,t_i)=h_{i,0}(x_i)+h_{i,1}(x_i)t_i+h_{i,2}(x_i)t_i^2+\cdots,\ (x_i,t_i)\in U_i\times\mc,
\end{align*}
where $h_{i,k}$, $k=0,1,2,\cdots$ are holomorphic functions on $U_i$.

Let $\{\varphi_{ij}\in\mathcal{O}^*(U_i\cap U_j)\}$ be the transition functions of the line bundle $L$,
we can infer that on $U_i\cap U_j$, $h_{i,k}=\frac{1}{\varphi_{ij}^k}h_{j,k}$.
This implies $\{h_{i,k}\in\mathcal{O}(U_i)\}_{i\in I}$ corresponds to an element in $H^0(X,kL^*)$.
\end{proof}

\subsection{Key lemmas for the proof of the main theorem}
We establish two lemmas that are used in the proof of Theorem \ref{thm:main}.
There results have their independent interest.
One key observation is that, if $L$ is negative, 
then generators of $R(X,L^*)$ (viewed as holomorphic functions on $L$) gives the Grauert blow-down of $L$
(it is well known that $R(X,L^*)$ is finitely generated as a $\mc$-algebra when $L^*$ is ample).

\begin{lem}\label{lem:generators give blow-down}
Let $L$ be a negative line bundle over a compact complex space $X$, 
and $\{f_1,\cdots,f_N\}$ be a set of generators of $R(X,L^*)$ as a $\mc$-algebra, 
then 
$$f=(f_1,\cdots,f_N):L\ra Z:=f(L)(\subset\mc^N)$$ 
induces the Grauert blow-down of $L$.
Moreover, if $\{f_1^{\prime},\cdots,f_M^{\prime}\}$ is another set of generators of $R(X,L^*)$,
then $f^{\prime}:L\ra Z':=f'(L)(\subset\mc^M)$ is isomorphic to $f:L\ra Z$.
\end{lem}
\begin{proof}
Firstly, we prove the second statement, that is there exists a commutative diagram
\[\begin{tikzcd}
	& {L} \\
	\\
	{Z} && {Z^{\prime}}
	\arrow["f"', from=1-2, to=3-1]
	\arrow["f^{\prime}", from=1-2, to=3-3]
	\arrow["\cong"', from=3-1, to=3-3].
\end{tikzcd}\]
Since $f_1,\cdots,f_N$ generate $R(X,L^*)$, 
there exists polynomials $p_j^{\prime}\in\mc[x_1,\cdots,x_N]$ such that 
$$f_j^{\prime}=p_j^{\prime}(f_1,\cdots,f_N),\ j=1,\cdots,M.$$
As a result, $p=(p_1,\cdots,p_M):\mc^N\ra \mc^M$ maps $Z$ onto $Z^{\prime}$ and $f^{\prime}=p^{\prime}\circ f $.
On the other hand, $\{f_j^{\prime}\}$ are generators of $R(X,L^*)$ as well, 
there is some polynomial  $p:\mc^M\ra\mc^N$ such that $f=p\circ f^{\prime}$.
This gives the above commutative diagram as we want.

Secondly, we prove that $f:L\ra Z$ is the Grauert blow-down of $L$.
Let $k\geq 1$ such that $kL^*$ is very ample and $(k+1)L^*$ is globally generated.
We can choose generators $\{f_1,\cdots,f_N\}$ of $R(X,L^*)$ satisfying
\begin{itemize}
\item[$\bullet$] 
$f_1,\cdots,f_{N_1}$ form a basis of $H^0(X,kL^*)$,
\item[$\bullet$]
$f_{N_1+1},\cdots,f_{N_2}$ form a basis of $H^0(X,(k+1)L^*)$.
\end{itemize}
It suffices to show that $f$ satisfies the following properties:
\begin{itemize}
\item[(1)]
$f(X)$ is a single point, $f:L\ra\mc^N$ is proper (by Remmert's proper mapping theorem, $Z:=f(L)$ is then an analytic subset of $\mc^N$),
\item[(2)]
$f:L\backslash X\ra\mc^{N}$ is a locally biholomorphism onto its image,
\item[(3)]
$f:L\backslash X\ra \mc^{N}$ is injective,
\item[(4)]
if $U\subseteq Z$ is open, $V=f^{-1}(U)$, then for every function $h$ holomorphic in $V$, there is a holomorphic function $g$ holomorphic in $U$ so that $h=g\circ f$.
\end{itemize}

\emph{Proof of $(1)$.}
We see  $f(X)=\{0\}$ since each $f_j$ vanishes identically on $X$.
To prove the properness of $f:L\ra\mc^N$, we take $\{U_i\}_{i=1}^M$ to be an open covering of $X$, 
by Lemma \ref{lem:expansion},
 $f_j(x,t_i)=f_{j,k}(x)t_i^k$ on $L|_{U_i}=U_i\times \mc$, $j=1,\cdots,N_1$. 
Thanks to $kL^*$ is globally generated, 
$f_{j,k}$ do not vanish simultaneuously,
the properness of $f$ follows.

\emph{Proof of $(2)$.}
We show that $\hat{f}=(f_1,\cdots,f_{N_1}):L \ra\mc^{N_1}$ does the job.
That is, $\hat{f}:L \backslash X\ra\mc^{N_1}$ is a locally biholomorphism onto its image.
Choose a trivialization $L|_U=U_x\times\mc_t$, 
$f_j(x,t)=f_{j,k}(x)t^k$ for some $f_{j,k}\in\mathcal{O}(U)$.
By shrinking $U$, 
we may assume $f_{1,k}$ does not vanish on $U$, 
then we discover
\begin{align*}
\tilde{f}:U\times \mc^*\ra\mc^{N_1},(x,t)\mapsto \Big(f_{1,k}(x)t^k,\frac{f_{2,k}(x)}{f_{1,k}(x)},\cdots,\frac{f_{N_1,k}(x)}{f_{1,k}(x)}\Big)
\end{align*}
is locally biholomorphic.
This is because $kL^*$ is very ample, 
the mapping $X\ra \mathbb{P}^{N_1-1},x\mapsto[f_1(x):\cdots:f_{N_1}(x)]$ is an embedding and for $t_0\neq0$, 
$\frac{\partial}{\partial t}(f_{1,k}(x)t^{k})|_{t=t_0}=k f_{1,k}(x)t_0^{k-1}\neq0$.
As a consequence, we can infer
\begin{align*}
\hat{f}:U\times \mc^*\ra\mc^{N_1},(x,t)\mapsto \Big(f_{1,k}(x)t^k, f_{2,k}(x)t^k,\cdots, f_{N_1,k}(x)t^k\Big)
\end{align*}
is also locally biholomorphic.

\emph{Proof of $(3)$}.
$kL^*$ is very ample, its global sections seperate points of $X$, 
therefore $\hat{f}$ maps distinct fibers of $L$ onto distinct complex lines passing through the origin.
For every $x\in X$, 
if we identify $\hat{f}(L_x)$ with $\mc$, the restriction $\hat{f}|_{L_x}$ has the form $\mc\ra\mc,t\mapsto t^k$.
Furthermore, 
since $(k+1)L^*$ is globally generated, we see that
\begin{align*}
(f_{N_1+1},\cdots,f_{N_2}):L_x\ra\mc^{N_2-N_1},t\mapsto(c_{N_1+1}t^k,\cdots,c_{N_2}t^k)
\end{align*}
with some of $c_i$ does not vanish.
As a consequence, $f=(f_1,\cdots,f_{N_2})$ is injective on $L\backslash X$, this proves $(3)$.

\emph{Proof of $(4)$.}
From $(1),(2),(3)$, we have $f:L \backslash X\ra Z\backslash\{0\}$ is biholomorphic.
To prove $(4)$, we only need to prove every holomorphic function $h$ in a neighborhood of the zero section $X$, 
there is a holormophic function $g\in\mathcal{O}_{Z,0}$ so that $h=g\circ f$.
By Lemma \ref{lem:expansion}, we expand $h$ into series:
\begin{align*}
h=\sum_{k=0}^{\infty}h_k, \ h_k\in H^0(X,kL^*),
\end{align*}
Since $f_1,\cdots,f_N$ generates $R(X,L^*)$, for every $k$, there is a polynomial $p_k\in\mc[x_1,\cdots,x_N]$ so that $h_0+\cdots+h_k=p_k(f_1,\cdots,f_N)$,
that is $h_0+\cdots+h_k=p(z_1,\cdots,z_N)\circ f$.
Since $h_0+\cdots+h_k$ converges locally uniformly to $h$ in a neighborhood of $X$ in $L$, 
we can infer $p_k(z_1,\cdots,z_N)$ converges locally uniformly in a neighborhood of $0$ in $Z$,
then Weierstrass uniformly convergence theorem for complex spaces (see e.g. \cite{Gun90}) implies $p_k$ converges to a function $g\in\mathcal{O}_{Z,0}$. 
Clearly $h=g\circ f$, we complete the proof.
\end{proof}

\emph{Remark.}
Let $\pi:(L,X)\ra (Z,z_0)$ be the Grauert blow-down of $L$, then any $g\in R(X, L^*)$ induces a holomorphic function on $Z$.
From the above lemma, we can see that a set of generators $\{f_1,\cdots, f_N\}$ of $R(X, L^*)$ always induces a set of generators 
of $\mathfrak{m}_{z_0}$ as an ideal of $\mathcal O_{z_0}$.

\begin{lem}\label{lem:invertble}
Let $L$ be a negative line bundle on a compact complex space $X$ and $f:(L,X)\ra (Z,z_0)$ be its Grauert blow-down.
Let $\mathfrak{m}_{z_0}$ be the maximal ideal of $\mathcal O_{Z,z_0}$,
then the ideal sheaf $f^{-1}(\mathfrak{m}_{z_0})\mathcal{O}_L$ is invertible if and only if $k_0L^*$ is globally generated, 
where $k_0$ is the initial order of $L^*$.
\end{lem}

Recall that a coherent ideal sheaf over a complex space is called invertible  if it is locally generated by a single element which is not zero-divisor.

\begin{proof}
For simplicity, we denote $f^{-1}(\mathfrak{m}_{z_0})\mathcal{O}_L$ just by $\mathfrak{m}_{z_0}\mathcal{O}_L$.
We can choose generators $\{f_1,\cdots,f_N\}$ of $R(X,L^*)$ so that
\begin{itemize}
\item[$\bullet$]
$\{f_1,\cdots,f_{N_0}\}$ form a basis of $H^0(X,k_0L^*)$,
\item[$\bullet$]
$f_i\in \bigoplus_{k=k_0+1}^{\infty}H^0(X,kL^*)$ for $i=N_0+1,\cdots,N$.
\end{itemize}
By Lemma \ref{lem:generators give blow-down}, $f=(f_1,\cdots,f_N):L\ra \mc^N$ gives the Grauert blow-down of $L$.

Suppose $k_0L^*$ is globally generated, 
we need to prove $\mathfrak{m}_{z_0}\mathcal{O}_L=f_1\mathcal{O}_L+\cdots+f_N\mathcal{O}_L$ is invertible.
Take some $x_0\in X$, we may assume the section $f_1$ does not vanish at $x_0$.
Then in a neighborhood $U$ of $x_0$ so that $L|_{U}=U_x\times\mc_t$,
we have
\begin{align*}
&f_1(x,t)=f_{1,k_0}(x)t^{k_0},\\
&\cdots\\
&f_{N_0}(x,t)=f_{N_0,k_0}(x)t^{k_0},\\
&f_{N_0+1}(x,t)=f_{N_0+1,k_0+1}(x)t^{k_0+1}+\cdots,\\
&\cdots
\end{align*}
Since $f_{1,k_0}$ does not vanish at $x_0$, we get $\frac{f_i}{f_1}$, $i\geq2$ are holomorphic in a neighborhood of $x_0$ in $L$.
Hence $\mathfrak{m}_{z_0}\mathcal{O}_L$ is invertible as desired.

Conversely, if $\mathfrak{m}_{z_0}\mathcal{O}_L=f_1\mathcal{O}_L+\cdots+f_N\mathcal{O}_L$ is invertible, we need to show $k_0L^*$ is globally generated.
Let $X=X_1\cup\cdots\cup X_s$ be the irreducible decomposition of $X$,
we may assume the sections $f_1,\cdots,f_{N_0}$ do not simultaneously vanish on $X_1$ identically. 
Since $X$ is connected, 
we can also assume that $X_{i+1}\cap(X_1\cup\cdots\cup X_i)$ is not empty for $i=1,\cdots, s-1$.

For every $x_1\in X_1$, the ideal $(f_1,\cdots,f_N)_{x_1}$ is generated by a non-zero divisor, 
thus we can infer $(f_1,\cdots,f_N)_{x_1}$ is generated by some $f_{i_0}$.
We claim that $i_0\leq N_0$.
Otherwise for some $i\leq N_0$, the section $f_{i}$ do not vanish identically on the germ $(X_1,x_1)$, 
which leads to $\frac{f_{i}}{f_{i_0}}$ could not be holomorphic at $x_1$, a contradiction.
Since the zero set of $\mathfrak{m}_{z_0}\mathcal{O}_L$ is $X$,
we get to know $\{f_{i_0}=0\}\cap V=X\cap V$ for some neighborhood $V$ of $x_1$ in $L$, 
which implies the section $f_{i_0}$ does not vanish at $x_1$.
Thus $X_1$ does not meet the base locus of $k_0L^*$.
From this and $X_2\cap X_1$ is non-empty, 
we infer that the sections $f_1,\cdots,f_{N_0}$ do not simultaneously vanish on the whole $X_2$,
then we use the same arguments as above to prove $X_2$ does not meet the base locus of $k_0L^*$.
By induction, 
we conclude that $k_0L^*$ is globally generated.

\end{proof}

\subsection{Proof of the main theorem}
With the preparations discussed in the previous subsections, we can now give the proof of Theorem \ref{thm:main}.


\begin{thm}(=Theorem \ref{thm:main})\label{thm:blow-up vs blow-down}
Let $L$ be a negative line bundle over a compact complex space and 
$f:(L,X)\ra(Z,z_0)$ be the Grauert blow-down.
Then $f:L\ra Z$ is the quadratic transform of $Z$ with center $z_0$ if and only if $k_0L^*$ is very ample and $(k_0+1)L^*$ is globally generated, where $k_0$ is the initial order of $L^*$.
\end{thm} 
\begin{proof}
Assume  $\{f_1,\cdots,f_N\}$ be a set of generators of $R(X,L^*)$ satisfying the following conditions:
\begin{itemize}
\item[(1)]
$\{f_1,\cdots,f_{N_0}\}$ form a basis of $H^0(X,k_0L^*)$,
\item[(2)]
$\{f_{N_0+1},\cdots,f_{N_1}\}$ form a basis of $H^0(X,(k_0+1)L^*)$,
\item[(3)]
$f_i \in \bigoplus_{k=k_0+2}^{\infty}H^0(X,kL^*)$ for $i=N_1+1,\cdots,N$,
\end{itemize}
Lemma \ref{lem:generators give blow-down} yields that $f=(f_1,\cdots,f_N):L\ra\mc^N$ gives the Grauert blow-down of $L$. 
For every $x_0\in X$, 
choose a neighborhood $U$ of $x_0$ so that $L|_{U}=U_x\times\mc_t$.

By Lemma \ref{lem:expansion}, we can span $f_i$ near $(x_0,0)$ for $i=1,\cdots,N$ as follows:
\begin{align*}
&f_1(x,t)=f_{1,k_0}(x)t^{k_0},\\
&\cdots\\
&f_{N_0}(x,t)=f_{N_0,k_0}(x)t^{k_0},\\
&f_{N_0+1}(x,t)=f_{N_0+1,k_0+1}(x)t^{k_0+1},\\
&\cdots\\
&f_{N_1}(x,t)=f_{N_1,k_0+1}(x)t^{k_0+1},\\
&f_{N_1+1}(x,t)=f_{N_1+1, k_0+2}(x)t^{k_0+2}+f_{N_1+1, k_0+3}(x)t^{k_0+3}+\cdots\\
&\cdots.
\end{align*}

\emph{Proof of ``only if".}
If $f:L\ra Z$ is the quadratic transform of $Z$ with the center $z_0$, then  $\mathfrak{m}_{z_0}\mathcal{O}_L$ is invertible.
It follows from Lemma \ref{lem:invertble} that $k_0L$ is globally generated.
From the above local expression of $f_i$,
we see that the mapping 
$$\Phi:L\ra\mathbb{P}^{N-1},\ p\mapsto[f_1(p):\cdots:f_N(p)]$$ 
is well-defined, thus we obtain the following commutative diagram
\[\begin{tikzcd}
	L &&&& {Z\times\mathbb{P}^{N-1}} \\
	\\
	&& Z
	\arrow["{F:=(f,\Phi)}", from=1-1, to=1-5]
	\arrow["f"', from=1-1, to=3-3]
	\arrow["\mathrm{pr}", from=1-5, to=3-3]
\end{tikzcd}\]
Since $f$ is the blow-up of $Z$ with center $\mathfrak{m}_{z_0}$, 
$F$ must give an isomorphism between $L$ and the closure of
\begin{align*}
\{(z,[w])\in ( Z\backslash\{z_0\}) \times\mathbb{P}^{N-1}: z_iw_j=z_jw_i,1\leq i,j\leq N \}
\end{align*}
in $Z\times\mathbb P^{N-1}$.
In particular,
\begin{align*}
F\big|_{X}:X\ra \{z_0\}\times \mathbb{P}^{N-1}, 
x\mapsto \big(z_0,[f_1(x):\cdots:f_{N_0}(x):0:\cdots:0]\big)
\end{align*}
is an embedding, 
hence $k_0L^*$ is very ample. 

We prove $(k_0+1)L^*$ is globally generated.
If $k_0=1$ there is nothing to do, 
so we assume $k_0\geq2$.
Suppose to the contrary that some $x_0\in X$ is a base point of $(k_0+1)L^*$. 
Take a neighborhood $U$ of $x_0$ and span $f_i$ in $L|_U\cong U\times\mc$ as before,
we may assume $f_{1,k_0}$ does not vanish at $x_0$, 
then $F$ maps $\mathcal{U}$, a neighborhood of $(x_0,0)$, into $Z\times\{[w_1:\cdots:w_N]\in\mathbb{P}^{N-1}: w_1\neq0\}$:
\begin{align*}
F: \mathcal{U}&\ra Z\times\mathbb{P}^{N-1}\\
p&\mapsto\Big(f_1(p),\cdots,f_N(p),[1:\frac{f_2}{f_1}:\cdots:\frac{f_{N_0}}{f_1}:\frac{f_{N_0+1}}{f_1}:\cdots:\frac{f_{N_1}}{f_1}:\cdots](p)\Big)
\end{align*} 
Since $k_0\geq2$, $\frac{\partial f_i}{\partial t}\big|_{(x_0,0)}=0$ for every $i=1,\cdots,N$.
Note for $i=2,\cdots,N_0$, $\frac{f_i}{f_1}$ are functions of $x$-variable only, 
hence $\frac{\partial}{\partial t}(\frac{f_i}{f_1})\big|_{(x_0,0)}=0$.
For $i=N_0+1,\cdots,N_1$,
we have $\frac{f_{i}}{f_1}=\frac{f_{i,k_0+1}}{f_{1,k_0}}\cdot t$,
the assumption that $x_0$ is a base point of $(k_0+1)L^*$ implies $\frac{f_{i,k_0+1}}{f_{1,k_0}}(x_0)=0$, 
hence $\frac{\partial  }{\partial t}(\frac{f_i}{f_1})\big|_{(x_0,0)}=0$.
Moreover, for $i=N_1+1,\cdots,N$, 
we have $\frac{f_i}{f_1}=O(t^2)$, 
hence $\frac{\partial}{\partial t}(\frac{f_i}{f_1})\big|_{(x_0,0)}=0$.
As a consequence, the restriction of $F$ on the fiber $L_{x_0}$ is not an embedding, a contradiction.
Therefore $(k_0+1)L^*$ is globally generated and we complete the proof of ``only if" part.

\emph{Proof of ``if".} 
Since $k_0L^*$ is globally generated, 
$$\Phi:L\ra\mathbb{P}^{N-1},p\mapsto[f_1(p):\cdots:f_N(p)]$$ 
is a well-defined holomorphic map,
we obtain the following commutative diagram
\[\begin{tikzcd}
	L &&&& {Z\times\mathbb{P}^{N-1}} \\
	\\
	&& Z
	\arrow["{F:=(f,\Phi)}", from=1-1, to=1-5]
	\arrow["f"', from=1-1, to=3-3]
	\arrow["\mathrm{pr}", from=1-5, to=3-3]
\end{tikzcd}\]
We only need to prove $F$ is a closed embedding.
The assumption that $k_0L^*$ is very ample makes $F|_X:X\ra \{z_0\}\times\mathbb{P}^{N-1}$ is a closed embedding, 
also note that $f:L\backslash X\ra Z\backslash\{z_0\}$ is biholomorphic,
it is sufficient to prove $F$ is an embedding at every point $x_0\in X$.

We take a neighborhood $U$ of $x_0$ so that $L|_{U}=U_x\times\mc_t$
and span $f_i$, $i=1,\cdots,N$ as before.
Since our complex space $X$ is not assumed to be regular, we need to
take a minimal embedding $\tau:U\hookrightarrow D$, 
where $D\subseteq\mc^d_z$ is the unit ball and $d=\dim_{\mc} \mathfrak{m}_{x_0}/\mathfrak{m}_{x_0}^2$.
Then $f_i$ extend to holomorphic functions on $D\times\mc$ which we shall still denote by $f_i$.
Assuming $f_{1,k_0}(x_0)\neq0$,
then $(\frac{f_2}{f_1},\cdots\frac{f_{N_0}}{f_1}):D\ra\mc^{N_0-1}$ is an embedding at $x_0$ (see e.g. \cite[Proposition 2.4]{Fis76}).
As a result, the matrix
$$
\begin{pmatrix}
\frac{\partial }{\partial z_1}(\frac{f_2}{f_1})& \cdots & \frac{\partial }{\partial z_1}(\frac{f_{N_0}}{f_1})\\
\vdots& \vdots& \vdots\\ 
\frac{\partial  }{\partial z_d}(\frac{f_2}{f_1})&  \cdots & \frac{\partial }{\partial z_d} (\frac{f_{N_0}}{f_1})
\end{pmatrix}
$$ 
has rank $d$ at $x_0$.

The Jacobian of the map
 $$F=(f_1,\cdots,f_N,\frac{f_2}{f_1},\cdots,\frac{f_N}{f_1}):D\times\mc\ra\mc^{2N-1}$$ at $(x_0,0)$ is equal to the following matrix
$$
\begin{pmatrix}
O & \frac{\partial }{\partial t}(\frac{f_2}{f_1}) & \cdots & \frac{\partial }{\partial t}(\frac{f_{N_0}}{f_1})& \frac{\partial }{\partial t}(\frac{f_{N_0+1}}{f_1}) & \cdots & \frac{\partial}{\partial t}(\frac{f_{N_1}}{f_1}) & \cdots & \frac{\partial }{\partial t}(\frac{f_N}{f_1})\\ 
	  O & \frac{\partial }{\partial z_1}(\frac{f_2}{f_1}) & \cdots & \frac{\partial }{\partial z_1} (\frac{f_{N_0}}{f_1})& \frac{\partial }{\partial z_1}(\frac{f_{N_0+1}}{f_1}) & \cdots&\frac{\partial }{\partial z_1}(\frac{f_{N_1}}{f_1}) & \cdots & \frac{\partial}{\partial z_1}(\frac{f_N}{f_1})\\ 
	\vdots&\vdots& \vdots& \vdots& \vdots& \vdots& \vdots& \vdots& \vdots\\ 
	O & \frac{\partial }{\partial z_d}(\frac{f_2}{f_1}) & \cdots & \frac{\partial }{\partial z_d}(\frac{f_{N_0}}{f_1}) & \frac{\partial }{\partial z_d}(\frac{f_{N_0+1}}{f_1}) & \cdots &\frac{\partial }{\partial z_d}(\frac{f_{N_1}}{f_1}) & \cdots & \frac{\partial }{\partial z_d}(\frac{f_N}{f_1})
\end{pmatrix}
$$
Since $(k_0+1)L^*$ is globally generated,
some of $\frac{\partial }{\partial t }(\frac{f_{N_0+1}}{f_1}),\cdots,\frac{\partial }{\partial t}(\frac{f_{N_1}}{f_1})$ does not vanish at $(x_0,0)$. 
Moreover,
$\frac{f_{N_0+1}}{f_1},\cdots,\frac{f_{N_1}}{f_1}$ have factor $t$,
hence the sub-matrix
$$
\begin{pmatrix}
\frac{\partial }{\partial z_1}(\frac{f_{N_0+1}}{f_1})& \cdots & \frac{\partial }{\partial z_1}(\frac{f_{N_1}}{f_1})\\
\vdots& \vdots& \vdots\\ 
\frac{\partial }{\partial z_d}(\frac{f_{N_0+1}}{f_1})&  \cdots & \frac{\partial }{\partial z_d}(\frac{f_{N_1}}{f_1})
\end{pmatrix}
$$
vanishes at $(x_0,0)$. 

As a consequence, the Jacobian of $F$ at $(x_0,0)$ has rank $(d+1)$, 
$F$ is an embedding at $(x_0,0)$.
The proof is complete.
\end{proof}

\section{Theta characteristic of compact Riemannian surfaces}\label{sec:theta char}
In this section, we introduce a type of line bundles over compact Riemannian surfaces that satisfy 
the condition in Theorem \ref{thm:main} but are not very ample. 

Let $X$ be a compact Riemann surface, $K$ be the canonical divisor of $X$.
A divisor $D$ is called special if $D$ and $K-D$ are effective.
We recall the Clifford theorem.
\begin{thm}(see e.g. \cite[Chap 2]{GH78})
For a special divisor $D$ of degree $d$ on a compact Riemann surface $X$, we have
\begin{align*}
h^0(X,D):=\text{dim}H^0(X,\mathcal O(D))\leq 1+\frac{d}{2},
\end{align*}
with equality holding only if $D=0,K$, or $X$ is hyperelliptic.
\end{thm}
A theta characteristic of $X$ is a divisor $D$ on $X$ such that $2D=K$. 
In literatures, $D$ is called even (resp. odd) if $h^0(D)\equiv0 \mod 2$ (resp. $h^0(D)\equiv 1 \mod 2$).
We refer readers to \cite{Ha82},\cite{Mum71} for more discussions of theta characteristics.

Now we assume $X$ is not hyperelliptic and genus $g=3$.
Suppose a theta characteristic $D$ is effective, 
Clifford theorem implies $h^0(D)<2$, so $D$ must be odd. 
This means every even theta characteristic $D$ is not effective.
Moreover, $X$ is non-hyperelliptic implies $2D=K$ is very ample, $\mathrm{deg}(3D)=3g-3= 2g$ implies $3D$ is globally generated. 
Let $L=L_{-D}$ be the associated line bundle of $-D$, then $L$ is negative. 
Theorem \ref{thm:main} implies that $f:(L,X)\ra (Z,z_0)$ is the quadratic transfrom of $Z$ with center $z_0$, but $L$ is not very ample.

In general, let $\mathscr{M}_g$ be the moduli space of compact Riemann surface of genus $g$, $\mathscr{M}_g^r$ be the subspace of $\mathscr{M}_g$ consisting of compact Riemann surfaces of genus $g$ which contain a theta characteristic $D$ with $h^0(D)\geq r+1$, $h^0(D)=r+1\mod2$.
We have
\begin{thm}(\cite[Theorem 2.17]{Bi87})
If $g\geq3$, $\mathscr{M}_g^1$ has pure codimension one in $\mathscr{M}_g$.
\end{thm}
For $X\in \mathscr{M}_g\backslash\mathscr{M}_g^1$, there exists theta characteristic $D$ which is not effective. 
If $X$ is not hyperelliptic, then $2D=K$ is very ample, $3D$ has degree $3g-3\geq 2g$ hence is globally generated.
The line bundles associated such theta characteristics satisfy the condition in Theorem \ref{thm:main},
but they are not very ample.

	\end{document}